\documentclass[12pt]{amsart}

\newtheorem{theorem}{Theorem}[section]

\newtheorem{corollary}[theorem]{Corollary}

\newtheorem{lemma}[theorem]{Lemma}

\theoremstyle{definition}
\newtheorem{definition}[theorem]{Definition}
\newtheorem{example}[theorem]{Example}

\theoremstyle{remark}
\newtheorem{remark}[theorem]{Remark}

\numberwithin{equation}{section}

\newcommand{\p}{\partial}

\newcommand{\sig}{\Sigma}

\newcommand{\n}{\nabla}

\newcommand{\la}{\langle}
\newcommand{\ra}{\rangle}

\newcommand{\der}[1]{\frac{\partial}{\partial #1}}
\newcommand{\fder}[2]{\frac{\partial #1}{\partial #2}}

\newcommand{\vp}{\varphi}

\begin{document}
\date{}
\title[The first Steklov eigenvalue, conformal geometry, and minimal surfaces]
{The first Steklov eigenvalue, conformal geometry, and minimal surfaces}
\author{Ailana Fraser}
\address{Department of Mathematics \\
                 University of British Columbia \\
                 Vancouver, BC V6T 1Z2}
\email{afraser@math.ubc.ca}
\author{Richard Schoen}
\address{Department of Mathematics \\
                 Stanford University \\
                 Stanford, CA 94305}
\thanks{The first author was partially supported by  the 
Natural Sciences and Engineering Research Council of Canada (NSERC) and the second
author was partially supported by NSF grant DMS-0604960}
\email{schoen@math.stanford.edu}

\begin{abstract} 
We consider the relationship of the geometry of compact Riemannian manifolds with boundary to the first nonzero eigenvalue $\sigma_1$ of the Dirichlet-to-Neumann map (Steklov eigenvalue). For surfaces $\Sigma$ with genus $\gamma$ and $k$ boundary components we obtain the upper bound 
$\sigma_1L(\p\Sigma)\leq 2(\gamma+k)\pi$. For $\gamma=0$ and $k=1$ this result was obtained by Weinstock in 1954,
and is sharp. We attempt to find the best constant in this inequality for annular surfaces ($\gamma=0$ and $k=2$). For rotationally symmetric metrics we show that the best constant is achieved by the induced metric on the portion of the catenoid centered at the origin which meets a sphere orthogonally
and hence is a solution of the free boundary problem for the area functional in the ball. For a general class of (not necessarily rotationally symmetric) metrics on the annulus, which we call supercritical,
we prove that $\sigma_1(\sig)L(\p\Sigma)$ is dominated by that of the critical catenoid with equality
if and only if the annulus is conformally equivalent to the critical catenoid by a conformal
transformation which is an isometry on the boundary. Motivated by the annulus case, we show that
a proper submanifold of the ball is immersed by Steklov eigenfunctions if and only if it is a free boundary solution. We then prove general upper bounds for 
conformal metrics on manifolds of any dimension which can be properly conformally immersed into the unit ball in terms of certain conformal volume quantities. We show that these bounds are
only achieved when the manifold is minimally immersed by first Steklov eigenfunctions. We also use
these ideas to show that any free boundary solution in two dimensions has area at least $\pi$, and 
we observe that this implies the sharp isoperimetric inequality for free boundary solutions in the two dimensional case. 
\end{abstract}

\maketitle

\section{Introduction}

In this paper we consider a spectral problem for manifolds $\Sigma$ with nonempty boundary 
$\p\Sigma$. It is well known that the Dirichlet-to-Neumann map which takes a function $u$ on 
$\p\Sigma$ to the normal derivative of the harmonic extension of $u$ is a self-adjoint operator with a discrete spectrum $\sigma_0=0<\sigma_1\leq \sigma_2\leq\cdots$ tending to infinity. The eigenvalues for this problem were first discussed in 1902 by Steklov and are often called Steklov eigenvalues. In 1954 Weinstock \cite{W} showed that for a simply connected plane domain $\Sigma$ the quantity 
$\sigma_1\cdot L(\p\Sigma)$ is uniquely maximized by a round disk. This theorem and its proof are analogous to a result of Szeg\"o \cite{Sz} for the first nonzero Neumann eigenvalue. Since that time there have been many papers written which give estimates on Steklov eigenvalues in higher dimensions and on Riemannian manifolds (see Payne \cite{P}, Bandle \cite{Ba}, Hersch and Payne \cite{HP}, Hersch, Payne, and Schiffer \cite{HPS}, Kuttler and Siggilito \cite{KS}, Shamma \cite{Sh}, Edward \cite{Ed}, Escobar \cite{E1}, \cite{E2}, Brock \cite{Br}, Dittmar \cite{D}, Girouard and Polterovich \cite{GP1}). The reader may consult Girouard and Polterovich \cite{GP2} for a recent survey which includes this topic.

Here we extend the result of Weinstock to arbitrary Riemannian surfaces with boundary to obtain the upper bound (Theorem \ref{theorem:yy})
\[ 
             \sigma_1 L(\p\Sigma)\leq 2(\gamma+k)\pi
\] 
for a surface of genus $\gamma$ with $k$ boundary components. Note that for $\gamma=0$ and $k=1$ this reduces to the Weinstock bound and is sharp. The proof uses a result of Ahlfors \cite{A}
with an improvement by Gabard \cite{G} to construct proper holomorphic maps from $\Sigma$ to the unit disk with controlled degree, and then the idea employed by Szeg\"o and Weinstock to use automorphisms of the disk to balance the map and construct test functions. Note that for compact surfaces without boundary an analogous idea was employed by Yang and Yau \cite{YY} to generalize a theorem of Hersch \cite{H} to higher genus surfaces.

We also show that the bound given above is not sharp at least for $\gamma=0$, and we attempt to determine the sharp bound for the case of surfaces homeomorphic to an annulus. An interesting elementary analysis 
(Section \ref{section:rotsym}) for annuli with rotationally symmetric metrics shows that for such annuli the quantity $\sigma_1 L(\p\Sigma)$ is maximized precisely for the ``critical catenoid"; that is, the portion of the catenoid centered at the origin inside the unit ball which meets the boundary orthogonally along the boundary of the ball. Moreover, for this surface, $\sigma_1$ has multiplicity $3$ and the eigenspace is spanned by the embedding functions. We denote $T(1)=2t_1$
where $t_1$ is the positive solution of $t_1=\coth(t_1)$. Numerically we see that $t_1\approx 1.2$. The critical catenoid is characterized by the conditions that it is biholomorphic to $[0,T(1)]\times S^1$, it has a rotationally symmetric metric, and it has boundary curves of equal length. The
value $(\sigma_1 L)^*$ of $\sigma_1 L$ for the critical catenoid is given by $4\pi/t_1$. 

Now for a general metric on an annulus it seems a reasonable conjecture that $\sigma_1 L$
should be at most $(\sigma_1 L)^*$. In Section 4 we obtain partial results in this direction.
Given any metric on an annulus $\Sigma$, we let $\alpha$ denote the ratio of its boundary lengths
and we let $T$ denote the unique positive real number for which $\Sigma$ is
biholomorphic to $[0,T]\times S^1$. We refer to those annuli with $T\leq 1/4(\alpha^{1/2}+\alpha^{-1/2})^2T(1)$ (resp. $T\geq 1/4(\alpha^{1/2}+\alpha^{-1/2})^2T(1)$) as {\it subcritical} (resp. {\it supercritical}). Thus any Riemannian annulus is either subcritical, supercritical, or both.

In Theorem \ref{theorem:supercritical} we show that for any supercritical Riemannian annulus we have $\sigma_1\cdot L(\p\Sigma)$ bounded above by $(\sigma_1 L)^*$, the corresponding quantity for the critical catenoid. We also show that equality is achieved if and only if $\Sigma$ is conformally equivalent to the critical catenoid by a conformal transformation which is an isometry on the boundary. Thus for supercritical annuli we obtain the improved sharp bound 
$\sigma_1 \cdot L(\p \sig) \leq 4\pi/t_1\approx 4\pi/1.2$ compared with the bound of
$4\pi$ given by Theorem \ref{theorem:yy}.

Motivated by the case of annuli we then explore the connection to minimal submanifolds $\Sigma^k$ lying in the unit ball $B^n$ with boundary contained in the boundary of the ball and with conormal vector equal to the position vector at boundary points. Such minimal submanifolds are critical for the free boundary problem of extremizing the volume among deformations which preserve the ball (but are not necessarily the identity along the boundary). We observe that such solutions arise from variational (min/max) constructions, and examples include equatorial disks, the critical catenoid discussed above, as well as the cone over any minimal submanifold of the sphere. Given a submanifold properly immersed in the unit ball, it is a free boundary solution if and only if the coordinate functions are Steklov eigenfunctions with eigenvalue $1$. It is then natural to ask whether free boundary solutions generally solve extremal problems for Steklov eigenvalues in natural classes of manifolds. In Section 
\ref{section:cv} we develop a theory which we call {\it boundary and relative conformal volume} because it is analogous to the conformal volume theory of Li and Yau \cite{LY} except that the boundary plays an essential role in this theory. Using the Gauss-Bonnet Theorem with boundary we show (Theorem \ref{theorem:max}) that when $k=2$, a free boundary solution has boundary length which is a maximum over the boundary lengths of its conformal images in the ball. We use this to show (Theorem \ref{theorem:lowerbound}) that any free boundary solution has area at least $\pi$.
We observe that this inequality is equivalent to the sharp isoperimetric inequality for free boundary surfaces. We define the boundary conformal volume to be the Li-Yau conformal volume of the boundary submanifold.

We then proceed to define a relative conformal volume for manifolds $\Sigma$ which admit proper conformal immersions into the unit ball. We take the maximum volume of the conformal images of a given immersion, and then minimize over conformal immersions. We show that the relative conformal volume gives a general upper bound on the first nonzero Steklov eigenvalue over all conformal metrics on $\Sigma$. Specifically we show for any $k$ the general upper bound on 
$\sigma_1 V(\p\Sigma)(V(\Sigma))^{(2-k)/2}$ in terms of the relative conformal volume. For $k=2$ this reduces to the bound $\sigma_1\cdot L(\p\Sigma)\leq 2 V_{rc}(\Sigma,n)$.

We thank the referee for several helpful suggestions, and especially for pointing out the recent
paper \cite{G} which gave an improvement of the bound in Theorem \ref{theorem:yy}. His 
questions also led to the formulation and inclusion of Theorem \ref{theorem:yysharp}.

\section{Dirichlet-to-Neumann Map} \label{section:dn}

Let $(\sig,g)$ be a compact $k$-dimensional Riemannian manifold with boundary 
$\p\sig \neq \emptyset$ and Laplacian $\Delta_g$. Given a function $u \in C^\infty(\p \sig)$, let $\hat{u}$ be the harmonic extension of $u$:
\[
    \begin{cases}
    \Delta_g \hat{u} =0  & \text{on }\; \sig, \\   
    \hat{u} =u &  \text{on } \p \sig.
    \end{cases}
\]
 Let $\nu$ be the outward unit conormal along $\p \sig$. The Dirichlet-to-Neumann map is the map
 \[
             L: C^\infty(\p \sig) \rightarrow C^\infty(\p \sig)
 \]
 given by 
 \[
           Lu=\fder{\hat{u}}{\nu}.
 \]         
$L$ is a nonnegative self-adjoint operator with discrete spectrum
$\sigma_0 < \sigma_1 \leq \sigma_2 \leq \cdots$ tending to infinity.
The eigenvalues for this problem were first discussed in 1902 by Steklov and are often called Steklov eigenvalues.

Since the constant functions are in the kernel of $L$, the lowest eigenvalue $\sigma_0$ of $L$ is zero. The first nonzero eigenvalue $\sigma_1$ of $L$ can be characterized variationally as follows:
\[
       \sigma_1=\inf_{u \in C^1(\p \sig),\; \int_{\p \sig} u=0}
       \frac{\int_\sig |\n \hat{u}|^2 \,da}{\int_{\p \sig} u^2 \,ds}.
\]

\begin{example} \label{example:ball}
As can be easily seen the eigenvalues of the Dirichlet-to-Neumann map of the unit ball $B^n$ in 
$\mathbb{R}^n$ are $\sigma_k=k$, $k=0, \,1, \, 2, \ldots$, and the eigenspace of $\sigma_k$ is given by the space of homogeneous harmonic polynomials of degree $k$ restricted to the sphere $\p B^n$.
\end{example}

We will say that a minimal submanifold $\sig$, properly immersed in a domain $\Omega$, is a
{\it free boundary solution} if the outward unit normal vector of $\p\sig$ (the conormal vector) agrees
with the outward unit normal to $\p\Omega$ at each point of $\p\sig$. This terminology
reflects the fact that such a minimal submanifold $\sig$ is a critical point of the volume functional
among relative cycles in $\Omega$; that is, $\sig$ is  critical for volume among deformations
which preserve the domain $\Omega$, but do not necessarily fix $\p\sig$.

In the next lemma we observe an interesting connection between eigenvalues of the Dirichlet-to-Neumann map, and minimal submanifolds that are solutions to the free boundary problem in the unit ball in $\mathbb{R}^n$.

\begin{lemma} \label{lemma:minimal}
Let $\sig^k$ be a properly immersed submanifold in the unit ball $B^n$ in $\mathbb{R}^n$.
Then $\sig^k$ is a minimal submanifold which meets $\p B^n$ orthogonally if and only if the coordinate
functions of $\sig$ in $\mathbb{R}^n$ are eigenfunctions of the Dirichlet-to-Neumann map, with eigenvalue one.
\end{lemma}

\begin{proof}
Let $\sig$ be a properly immersed submanifold in the unit ball $B^n$ with $\p \sig \subset \p B^n$.
Then it is well known that $\sig$ is minimal if and only if the coordinate functions $x^i$ of $\sig$ in 
$\mathbb{R}^{n}$ are harmonic functions, $\Delta x^i=0$ for $i=1, \ldots, n$. If $\nu$ is the outward unit conormal along $\p \sig$, the condition
\[
            \der{\nu} x^i=x^i, \qquad i=1, \ldots, n
\]
is equivalent to $\nu=x$, which is equivalent to the condition that $\sig$ meets $\p B^n$ orthogonally.
Therefore, $\sig$ is a solution to the free boundary problem if and only if the coordinate functions are eigenfunctions of the Dirichlet-to-Neumann map, with eigenvalue one.
\end{proof}

Lemma \ref{lemma:minimal} is analogous to the property that a submanifold $\sig^k$ immersed in the sphere $S^n$ is a minimal submanifold of the sphere if and only if the coordinate functions are eigenfunctions of the Laplacian with eigenvalue $k$.

\vspace{3mm}

The first result giving a bound on $\sigma_1$ was due to Weinstock \cite{W}, where he proved that if 
$\sig$ is a compact simply connected plane domain then
\[
       \sigma_1 L(\p \sig) \leq 2\pi,
\]
with equality if and only if $\sig$ is a disk. He went on to show that if $\Sigma$ is a simply connected surface with boundary, then $\sigma_1 L(\p \sig) \leq 2\pi$ with equality if and only if there is conformal map from $\Sigma$ to the unit disk which is an isometry on the boundary. Note that $\sigma_1(D)=1$ (as in example \ref{example:ball}), and so the estimate can be written 
\[
      \sigma_1(\sig) L(\p \sig) \leq \sigma_1(D) L(\p D).
\]
This type of eigenvalue estimate originated in Szeg\"o's estimate \cite{Sz} of the first Neumann eigenvalue of the Laplacian for a simply connected domain in $\mathbb{R}^2$.
In the following we extend the bound of Weinstock to arbitrary Riemannian surfaces with boundary.
This estimate for $\sigma_1$ is analogous to the estimates of Hersch \cite{H} and Yang-Yau \cite{YY} for the first eigenvalue of the Laplacian on surfaces.

\begin{theorem} \label{theorem:yy}
Let $\sig$ be a compact surface of genus $\gamma$ with $k$ boundary components. Let 
$\sigma_1$ be the first non-zero eigenvalue of the Dirichlet-to-Neumann operator on $\sig$ with metric $g$. Then
\[
       \sigma_1 L(\p \sig) \leq 2 (\gamma+k) \pi.
\]
\end{theorem}

\begin{remark}
When $\gamma=0$ and $k=1$ this reduces to Weinstock's estimate and is sharp.
\end{remark}

\begin{proof}
Any compact surface with boundary can be properly conformally branched over the disk  $D$. Specifically, there exists an Ahlfors function; that is, a proper conformal branched cover
$\vp: \sig \rightarrow D$, of degree at most $\gamma +k$. This originates in Ahlfors \cite{A} 
where it is shown that such a map exists of degree at most $2\gamma+k$. In the recent paper
of A. Gabard \cite{G} a new construction of such maps is given with the improved bound on
the degree. We may now use automorphisms of the disk to balance the map $\vp$. Specifically (see Lemma \ref{lemma:hersch}) we can assume that $\vp=(\vp^1, \vp^2)$ satisfies
\[
      \int_{\p \sig} \vp^i \; ds =0
\] 
for $i=1, 2$ where $ds$ is the element of arclength for $g$. Thus $\vp^1$ and $\vp^2$ can be used as test functions, so let $\hat{\vp}^i$ be the harmonic extension of $\vp^i |_{\p \sig}$. Then, by the variational characterization of $\sigma_1$,
\begin{equation} \label{eq:var-char}
     \sigma_1 \int_{\p \sig} (\vp^i)^2 \; ds 
                 \leq \int_\sig |\n \hat{\vp}^i |^2 \; da   
                 \leq \int_\sig |\n \vp^i |^2 \; da.   
\end{equation}
where $da$ is the area measure of $g$. Since $\vp$ is conformal, 
\[
      \sum_{i=1}^2 \int_\sig | \n \vp^i |^2 \; da 
      = 2 A( \vp(\sig)) \\
      \leq 2(\gamma +k) \pi.
\]
On the other hand, since $\vp$ is proper we have $\vp(\p \sig) \subset \p D$, so
\[
     \sum_{i=1}^2 \int_{\p \sig} (\vp^i)^2 \; ds = \int_{\p \sig} ds = L(\p \sig).
\]
Therefore (\ref{eq:var-char}) gives
\[
    \sigma_1 L(\p \sig) \leq 2(\gamma+k)\pi.
\]
\end{proof}
When the genus of the surface is zero, we show that the bound given in Theorem \ref{theorem:yy} is not sharp. In the next two sections we attempt to find a sharp bound in the case of surfaces homeomorphic to an annulus ($\gamma=0,\ k=2$).

\begin{theorem} \label{theorem:yysharp}
Let $\sig$ be a compact surface of genus $0$ with $k$ boundary components, $k \geq 2$. Let 
$\sigma_1$ be the first non-zero eigenvalue of the Dirichlet-to-Neumann operator on $\sig$ with metric $g$. Then
\[
       \sigma_1 L(\p \sig) < 2 k \pi.
\]
\end{theorem}
\begin{proof}
Let $\vp:\sig\to D$ be the proper conformal map of degree $k$ used in the proof of 
Theorem \ref{theorem:yy}. If it were true that $\sigma_1 L(\p \sig) =2 k \pi$, then
it would follow that the function $u=a\vp^1+b\vp^2$ is a first Steklov eigenfunction
for any real numbers $a,b$ which are not both $0$. The nodal set of this function
is the preimage under $\vp$ of the line $ax+by=0$. Now from the nodal domain
theorem for Steklov eigenfunctions which follows directly from the minmax characterization
of Steklov eigenvalues (see \cite{KS}) there can be at most two connected components
of $\sig\setminus\{u=0\}$. We complete the proof by showing that $a,b$ can be chosen
so that there are at least three components.

We first note that the map $\vp$ is a diffeomorphism from each boundary component
of $\sig$ to the unit circle $C$ since for any point $z\in C$ the set $\vp^{-1}(z)$ consists of exactly
$k$ points all of which lie on $\p\sig$. Since the map $\vp$ must be surjective from each
boundary component to $C$, it follows that exactly one of the points of $\vp^{-1}(z)$ lies
on each component of $\p\sig$. Note that this also shows that $k$ is the smallest possible
degree of a proper conformal map from $\sig$ to $D$.

To analyze the nodal domains of the function $u=a\vp^1+b\vp^2$, we observe that if there
are no branch points in the set $\Gamma=\{u=0\}$, then $\Gamma$ is a properly embedded
curve which is a union of $k$ segments each with endpoints on $\p\sig$. Moreover, each component 
of $\p\sig$ contains precisely two endpoints. A finite sequence of arcs of $\Gamma$ beginning and
ending on the same component of $\Gamma$ will be called a {\it circuit}; for example, an
arc with both endpoints on the same component of $\p\sig$ is a circuit of length $1$ while
a sequence going from a component $\sigma_1$ to $\sigma_2$ on to $\sigma_3$ and back to 
$\sigma_1$ would be a circuit of length $3$. Since $\sig$ has genus $0$ every circuit divides
$\sig$. If the circuit has length less than $k$, then one of the components has closure which
is disjoint from the components of $\p\sig$ which the circuit does not visit. We will call this
component the {\it small} component. In case there is a circuit of length less than $k$, there
is also a circuit among the remaining components of $\p\sig$, and the small components
of these two have disjoint closure. In particular it follows that $\sig\setminus\Gamma$ has
at least $3$ components contradicting the assumption that $u$ is  first Steklov eigenfunction.
Thus we have shown that the set $\Gamma$ is made up of a circuit of length $k$.

On the other hand, we may choose $a,b$ so that $\Gamma$ contains at least one
branch point. In this case, for $\epsilon\neq 0$ sufficiently small, the sets 
$\Gamma_\epsilon=\{u=\epsilon\}$ are free of branch points. We again refer to
an embedded curve $\gamma$ in $\Gamma$ as a circuit if $\gamma$ consists of a
sequence of embedded segments beginning and ending at the same component 
of $\p\sig$. Any circuit in $\Gamma_\epsilon$ converges as $\epsilon$ goes to zero
to a circuit in $\Gamma$. It follows as above that $\Gamma_\epsilon$ must contain
a circuit of length $k$. On the other hand, as $\epsilon$ goes to zero, two distinct
arcs from $\Gamma_{-\epsilon}$ and $\Gamma_\epsilon$ must come together at
a branch point. Without loss of generality, assume that the order of the circuit
$\Gamma_{-\epsilon}$ is given by $\sigma_1,\sigma_2,\ldots,\sigma_k$ and that the
arc from $\sigma_1$ to $\sigma_2$ joins at a branch point with the arc from $\sigma_i$ to 
$\sigma_j$ of $\Gamma_\epsilon$ where $i<j$. Now if $i$ is $1$ or $2$, then $\Gamma$
contains a circuit of length $1$, so we may assume that $i\geq 3$. Thus $j\geq 4$,
and we have a circuit in $\Gamma$ from $\sigma_1$ to $\sigma_2$ to $\sigma_j$
to $\sigma_{j+1}$ and in order back to $\sigma_1$. This circuit has length less than
$k$, and this contradicts the assumption that $u$ is a first Steklov eigenfunction. 
\end{proof}

\section{Rotationally symmetric  metrics on the annulus} \label{section:rotsym}

Consider a metric of the form $g=dr^2+a^2(r) d\theta^2$ where $0<r_1<r<r_2$, $\theta \in S^1$. Such a metric is isometric to a product $[0,T] \times S^1$  for some $T>0$, with metric 
$g=f^2(t)(dt^2+d\theta^2)$ for a positive function $f(t)$. Notice that a harmonic function $u(t,\theta)$ is harmonic with respect to the flat metric $dt^2+d\theta^2$ and thus satisfies the equation 
$u_{tt}+u_{\theta \theta}=0$. The outward unit normal vector on $\Gamma_0=\{t=0\}$ is given by 
$\nu=-f^{-1}(0) \der{t}$ and on $\Gamma_T=\{t=T\}$ by $\nu=f(T)^{-1}\der{t}$. On the other hand the arclength element is $f(0)d\theta$ at $t=0$ and $f(T)d\theta$ at $t=T$. Therefore
\[
    \int_{\p \sig} u u_\nu \;ds = \int_{\Gamma_T} u u_t \; d\theta
                                                  -\int_{\Gamma_0} u u_t \; d\theta.
\]
Finally the $L^2$ norm of $u$ on the boundary is given by 
\[
     \int_{\p \sig} u^2 \;ds = \int_{\Gamma_0} u^2 f(0) \; d\theta + \int_{\Gamma_T} u^2 f(T) \;d\theta.
\]
Thus there is no loss of generality for computing the eigenvalues of the Dirichlet-to-Neumann map by replacing $g$ by the flat metric given by choosing $f$ to be the linear function 
\[
       f(t)=\left(1-\frac{t}{T}\right)f(0) + \frac{t}{T}f(T).
\]
We will assume that our metric $g$ is of this form.

To compute the Dirichlet-to-Neumann spectrum we separate variables and look for harmonic functions of the form $u(t,\theta)=\alpha(t)\beta(\theta)$. By standard methods we obtain solutions for each nonnegative integer $n$ given by linear combinations of the functions $\sinh(nt)\sin(n\theta)$,
$\sinh(nt)\cos(n\theta)$, $\cosh(nt)\sin(n\theta)$, and $\cosh(nt)\cos(n\theta)$ if $n \geq 1$. For $n=0$ the solutions are linear combinations of the functions 1 and $t$. In order to be an eigenfunction for the Dirichlet-to-Neumann map we must have $u_\nu=\lambda u$ on $\p \sig$, or
$u_t=-\lambda f(0)u$ on $\Gamma_0$ and $u_t=\lambda f(T)u$ on $\Gamma_T$.
For $n=0$ we have $\alpha(t)=a+bt$ and the conditions become
$b=-\lambda f(0)a$, $b=\lambda f(T)(a+bT)$.
There are two values of $\lambda$ for which these have a nonzero solution: $\lambda^{(1)}_0=0$ with constant eigenfunctions and $\lambda^{(2)}_0=(f(0)f(T)T)^{-1}(f(0)+f(T))$ with eigenfunctions spanned by $1+bt$ where $b=-\lambda^{(2)}_0f(0)$. 

For $n \geq1$ the eigenfunctions have 
$\alpha(t)=a \sinh (nt) + b \cosh(nt)$ and the conditions are 
\begin{align*}
     na &=-\lambda f(0)b \\
     na \cosh(nT)+nb\sinh(nT)&=\lambda f(T)(a \sinh(nT)+b\cosh(nT)).
\end{align*}
Using the first equation to eliminate $a$ and dividing by $b$ (which must be nonzero in this case) we get the quadratic equation for $\lambda$
\[
      \lambda^2-n[f(0)^{-1}+f(T)^{-1}] \coth(nT) \lambda + n^2 f(0)^{-1}f(T)^{-1}=0.
\]
This equation has two positive roots $\lambda^{(1)}_n < \lambda^{(2)}_n$ given by
\begin{align*}
    \lambda^{(1)}_n&=\frac{n}{2}\left( [f(0)^{-1}+f(T)^{-1}] \coth(nT)- \sqrt{D} \right) \\
    \lambda^{(2)}_n&=\frac{n}{2}\left( [f(0)^{-1}+f(T)^{-1}] \coth(nT)+ \sqrt{D} \right)
\end{align*}
where
\[
    D = [ f(0)^{-1}+f(T)^{-1}]^2 \coth^2(nT)-4f(0)^{-1}f(T)^{-1}. 
\] 
Note that each of $\lambda^{(1)}_n$, $\lambda^{(2)}_n$ has multiplicity two. The expression for $\lambda^{(1)}_n$ can be rewritten
\[
\lambda^{(1)}_n=2nf(0)^{-1}f(T)^{-1}\left[(f(0)^{-1}+f(T)^{-1})\coth(nT)+\sqrt{D}\right]^{-1}.
\]
Since $\coth$ is decreasing for positive arguments, we can see that $\lambda^{(1)}_n$ is an increasing function of $n$. Thus if we want to find the smallest nonzero eigenvalue $\sigma_1$ of the Dirichlet-to-Neumann map we need only consider $n=0,\,1$. We must have either $\sigma_1=\lambda^{(2)}_0$ or 
$\sigma_1=\lambda^{(1)}_1$, and $\sigma_1=\min \{\lambda^{(2)}_0, \, \lambda^{(1)}_1 \}$. We consider the ratio $\lambda^{(1)}_1/\lambda^{(2)}_0$ and compute from the expression above
\begin{align*}
    \frac{\lambda^{(1)}_1}{\lambda^{(2)}_0}
    & =\frac{T}{2} \left[ \coth T - \left( \coth^2T -\frac{4f(0)f(T)}{(f(0)+f(T))^2}\right)^{1/2} \right]\\
    & =\frac{2Tf(0)f(T)}{(f(0)+f(T))^2}\left[\coth(T)+ \left( \coth^2T -\frac{4f(0)f(T)}{(f(0)+f(T))^2}\right)^{1/2}\right]^{-1}.
\end{align*}
If we let $\alpha=f(0)/f(T)$ denote the ratio of the boundary lengths, then
the this may be written
\[
    \frac{\lambda^{(1)}_1}{\lambda^{(2)}_0}
      =\frac{2T\alpha}{(\alpha+1)^2}\left[\coth(T)+ \left( \coth^2T -\frac{4\alpha}{(\alpha+1)^2}\right)^{1/2}\right]^{-1}.
\]
If we fix the value of $\alpha$, this expression is an increasing function of $T$ which tends to 0 as $T$ goes to 0 and to infinity as $T$ goes to infinity. Thus there is a unique $T(\alpha)>0$ such that $\sigma_1=\lambda^{(1)}_1$ for $T \leq T(\alpha)$ and $\sigma_1= \lambda^{(2)}_0$ for $T \geq 
T(\alpha)$. Thus for $T >T(\alpha)$ the multiplicity of $\sigma_1$ is one, for $T < T(\alpha)$ the multiplicity is two, and for $T=T(\alpha)$ the multiplicity is three. If we fix $f(0)$ and $f(T)$ we see that $\sigma_1$ is maximized for $T=T(\alpha)$ since $\lambda^{(2)}_0$ is a decreasing function of $T$ and $\lambda_1^{(1)}$ is an increasing function of $T$. It follows that if we fix $\alpha$,
then $\sigma_1L$ is maximized for $T=T(\alpha)$. 

The following theorem summarizes
the results for rotationally symmetric metrics. We give a geometric proof which identifies
the maximal annulus for a given $\alpha$ as a specific piece of the catenoid and shows that
$\sigma_1 L$ is maximized when $\alpha=1$. We consider the catenoid parametrized on
$(-\infty,\infty)\times S^1$ given by 
\begin{align*}
    x_1(t,\theta)&=\cosh t \cos \theta \\
    x_2(t,\theta)&=\cosh t \sin \theta \\
    x_3(t,\theta) &= t.
\end{align*}
The induced metric is $g=\cosh^2(t)(dt^2+d\theta^2)$, so the portion between $t=a$
and $t=b$ is conformally equivalent to $[a,b]\times S^1$.

\begin{theorem}\label{theorem:rotational}
Given any $a\in\mathbb{R}$, let $t_1$ and $t_2$ be the positive solutions of $t_1=\coth(t_1+a)$ and $t_2=\coth(t_2-a)$. If we let $\alpha=t_1/t_2$ then $\alpha$ is a decreasing function of
$a$ with range all of $\mathbb{R}_+$ and we have $T(\alpha)=t_1+t_2$ and the maximum value of 
$\sigma_1 L$ with $\alpha$ fixed is $2\pi(t_1^{-1}+t_2^{-1})$. Furthermore the maximum of 
$\sigma_1 L$ over all rotationally symmetric metrics on the annulus occurs uniquely when $a=0$, hence $\alpha=1$, and it is equal to $4\pi/t_1$ where $t_1=\cosh t_1$. The corresponding optimal value of $T$ is $T(1)=2t_1$. The numerical value of $t_1$ is approximately $1.2$.
\end{theorem}
\begin{proof}
The values $-t_1$ and $t_2$ are the values of $t$ at which the tangent line at $(t,x)$ to the graph of
$x=\cosh(t-a)$ passes through the origin. At such a point we would have $\cosh(t-a)/t=\sinh(t-a)$,
and we call the positive solution $t_2$ and the negative solution $-t_1$ with $t_1>0$. We then
have $-t_1=\coth(-t_1-a)=-\coth(t_1+a)$.

If we let $\Sigma$ be the portion of the surface of revolution gotten by revolving the graph
of $x=\cosh(t-a)$, then the choice of $t_1$ and $t_2$ guarantee than the conormal vector
of $\Sigma$ at each boundary component is parallel to the position vector. If we denote
by $\Gamma_1$ the boundary component corresponding to $-t_1$ and by $\Gamma_2$ the
boundary component corresponding to $t_2$, we then see that $\Gamma_1$
is contained in the sphere with center at the origin of radius $R_1=\sqrt{t_1^2+\cosh^2(t_1+a)}$
while $\Gamma_2$ is contained in the sphere of radius $R_2=\sqrt{t_2^2+\cosh^2(t_2-a)}$. Thus
if we take the induced metric $g$, then the coordinate functions $X=(x_1,x_2,x_3)$ are harmonic and satisfy the boundary conditions $\partial X/\partial \nu=(R_i)^{-1}X$ on $\Gamma_i$ where
$\nu$ denotes the unit conormal vector. It follows that if we rescale the metric near $\Gamma_2$
by a factor of $R_1/R_2$, then $(R_1)^{-1}$ is a first Steklov eigenvalue of multiplicity three.
The length of $\Gamma_1$ is $2\pi\cosh(t_1+a)$ while the length of $\Gamma_2$ in the
rescaled metric is $2\pi(R_1/R_2)\cosh(t_2-a)$. Thus the ratio $\alpha$ of the length of 
$\Gamma_2$ to that of $\Gamma_1$ is equal to the ratio of $\cosh(t_2-a)/R_2$ to
$\cosh(t_1+a)/R_1$. Using the definitions of the terms we see directly that $\cosh(t_1+a)/R_1=
t_1^{-1}$ and $\cosh((t_2-a)/R_2=t_2^{-1}$ and therefore $\alpha=t_1/t_2$. Based on the
discussion above we have shown that $T(\alpha)=t_1+t_2$. Now it is clear that $t_1$ is a
decreasing function of $a$ which tends to $\infty$ as $a$ tends to $-\infty$ and to $1$
as $a$ tends to $\infty$ while $t_2$ is an increasing function of $a$ which tends to $\infty$ as $a$ tends to $\infty$ and to $1$ as $a$ tends to $-\infty$. Thus we
see that $\alpha$ is a decreasing function which tends to $\infty$ as $a$ tends to $-\infty$ and 
to $0$ as $a$ tends to $\infty$. Now the value of $\sigma_1 L$ is given by $2\pi (R_1)^{-1}[\cosh(t_1+a)+(R_1/R_2)\cosh(t_2-a)]$, and this is equal to $2\pi(t_1^{-1}+t_2^{-1})$.

Finally we must show that the maximum of $\sigma_1 L$ occurs when $a=0$. From the
previous paragraph this amounts to showing that the function $f(a)=t_1^{-1}+t_2^{-1}$ is
maximized at $a=0$. From the definitions it is clear that $t_2(-a)=t_1(a)$, so it follows
that $f(-a)=f(a)$. Thus it suffices to show that $f(a)<f(0)$ for all $a>0$. To accomplish this
we compute $f'(a)=-t_1^{-2}t_1'(a)-t_2^{-2}t_2'(a)$. From the definition $t_1=\coth(t_1+a)$
we see that $t_1'(a)=(1-\coth^2(t_1+a))(t_1'(a)+1)$. It follows that $t_1'(a)=-(1-t_1^{-2})$
and similarly $t_2'(a)=1-t_2^{-2}$. Therefore we have 
$f'(a)=t_1^{-2}(1-t_1^{-2})-t_2^{-2}(1-t_2^{-2})$. If we let $Q(x)=x(1-x)$ then this is the same
as $f'(a)=Q(t_1^{-2})-Q(t_2^{-2})$. Notice that both $t_1$ and $t_2$ are greater than $1$,
so we are interested in values of $Q(x)$ with $0<x<1$. Now the value of $t_1(0)=t_2(0)$
is approximately $1.2$, so both $t_1^{-2}$ are $t_2^{-2}$ are greater than $1/2$ at $a=0$.
Since $t_1$ is decreasing in $a$ while $t_2$ is increasing it follows from the shape of the 
quadratic function $Q$ that $f'(a)<0$ at least until the value $a_0$ for which 
$t_2^{-2}<1-t_1^{-2}(0)$. Thus it follows that $f(a)<f(0)$ for $0<a<a_0$. Now for
$a\geq a_0$ we have $f(a)=t_1^{-1}+t_2^{-1}<1+\sqrt{1-t_1^{-2}(0)}$, and the result
follows from the inequality $1+\sqrt{1-t_1^{-2}(0)}\leq 2t_1^{-1}(0)$. This last inequality
is equivalent to $t_1(0)+\sqrt{t_1^2(0)-1}<2$, a numerical inequality for $t_1(0)$
which is easily verified.
\end{proof}

The extremal annulus which arises in the above theorem is a minimal surface which is
of geometric interest. It is the unique portion of the catenoid cut out by a ball centered
at the origin which meets the boundary orthogonally; that is, it is a free boundary solution
in the ball. We refer to this portion of the catenoid as the ``critical" catenoid.

\section{Supercritical Annuli}

In this section we consider the annulus with an arbitrary metric. Any such annulus $\sig$ is conformally equivalent to $[0,T] \times S^1$ with the metric $dt^2+d\theta^2$, for some unique $T>0$ called the {\it conformal modulus}. Let $\alpha$ be the ratio of boundary lengths (our statements and conclusions are
symmetric when $\alpha$ is replaced by $\alpha^{-1}$).
As in the previous section, we denote by $T(1)$ the conformal modulus of the critical catenoid.
In this section we consider the {\em supercritical case}, when $T \geq 1/4(\alpha^{1/2}+\alpha^{-1/2})^2T(1)$, and we show that the critical catenoid maximizes $\sigma_1 L$ over all supercritical metrics on the annulus. We let $(\sigma_1L)^*=8\pi/T(1)\approx 4\pi/1.2$ denote the value
of $\sigma_1 L$ for the critical catenoid.

\begin{theorem} \label{theorem:supercritical}
For any supercritical metric on an annulus we have
\[
       \sigma_1(\sig) L(\partial\Sigma) \leq (\sigma_1L)^*. 
\]
Moreover, equality is achieved if and only if 
$\sig$ is conformally equivalent to the critical catenoid by a conformal transformation which is an isometry on the boundary; in particular, $\sig$ has boundary components of equal length 
$(\alpha=1)$.
\end{theorem}

\begin{proof}
Let $\sig$ be a supercritical annulus. 
Then there is a conformal diffeomorphism 
\[
        F: \sig \rightarrow [0,T] \times S^1, 
\]
where $T \geq 1/4(\alpha^{1/2}+\alpha^{-1/2})^2T(1)$. 
Denote by $\Gamma_0$ and $\Gamma_1$ the boundary components of $[0,T]\times S^1$
and let $L_0$ and $L_1$ denote the boundary lengths relative to the metric of $\sig$. Construct  
$\tilde{\sig}=[0,T] \times S^1$ with the flat conformal metric with the same boundary lengths $L_0$ and $L_1$.

Now choose a nonconstant eigenfunction $\ell(t)$ which is linear in $t$, and is normalized to have boundary $L^2$ norm one, $||\ell||_{L^2(\p \tilde{\sig})}=1$. We may pull back $\ell$ to $\sig$. 
Then $||\ell \circ F||_{L^2(\p \sig)}=1$, and $\ell \circ F$ is constant on $\Gamma_0$ and 
$\Gamma_1$ since $\ell$ is a linear function of $t$. We then have
\begin{align*}
      \int_{\p \sig} (\ell \circ F) \; ds 
      &= \ell(0) L(\Gamma_0) + \ell(T) L(\Gamma_1) \\
      &= \ell(0) L(\tilde{\Gamma}_0) + \ell(T) L(\tilde{\Gamma}_1) \\
      &= \int_{\p \tilde{\sig}} \ell \; d\tilde{s}  \\
      &= 0,
\end{align*}
where the last equality follows since $\ell$ is an eigenfunction on $\tilde{\sig}$.
Therefore, we can use $\ell \circ F$ as a test function in the variational characterization of $\sigma_1$, and we obtain:
\[
      \sigma_1(\sig) \leq E(\ell \circ F) = E(\ell) = \lambda_0^{(2)},
\]
$\lambda_0^{(2)}$ is defined in the previous section. From the calculation of the previous section we have $\lambda^{(2)}_0=(f(0)f(T)T)^{-1}(f(0)+f(T))$ where $2\pi f(0)=L_0$ and $2\pi f(T)=L_1$. It follows that 
\[
\lambda^{(2)}_0 L(\p\tilde{\sig})=2\pi(L_0L_1T)^{-1}(L_0+L_1)^2=2\pi T^{-1}(\alpha^{1/2}+\alpha^{-1/2})^2\leq 8\pi/T(1)
\]
from the supercritical condition. Thus,
\begin{equation}   \label{equation:supercritical}
     \sigma_1(\sig) L(\p\sig) \leq \lambda_0^{(2)} L(\p\tilde{\sig}) 
     \leq 8\pi/T(1)=(\sigma_1L)^*,
\end{equation}
which is the required inequality.

Now consider the case of equality. We must have from above that $\sigma_1(\sig)= E(\ell \circ F)$
and therefore $\ell\circ F$ is a first Steklov eigenfunction on $\sig$. If we write the metric of
$\sig$ as $\lambda^2(dt^2+d\theta^2)$, this implies that $\p \ell / \p \nu=\sigma_1 \ell$ on the
boundary of $[0,T]\times S^1$ where $\nu =-\lambda^{-1} \p / \p t$ on $\Gamma_0$
and $\nu =\lambda^{-1} \p / \p t$ on $\Gamma_1$. Since $\ell$ and $\p\ell/\p t$ are constant on each boundary component, it follows that $\lambda$ is constant on each boundary component. Therefore by the normalization of the lengths of the boundary components of $\tilde{\sig}$ we see than $F$ is an isometry from $\p\sig$ to $\p\tilde{\sig}$. 

It remains to show that $\tilde{\sig}$ is a rotationally symmetric metric on the critical catenoid with equal boundary lengths. Since $\sigma_1(\sig)L(\p\sig)=(\sigma_1L)^*$, it follows from (\ref{equation:supercritical})  that $\lambda_0^{(2)} L(\p\tilde{\sig})= (\sigma_1L)^*$. We have shown that $F$ is an isometry on the boundary, so it follows that the Steklov eigenvalues of $\sig$ are the same as those of $\tilde{\sig}$. We have also shown that $\lambda_0^{(2)}$ is the first eigenvalue for $\sig$, and therefore it is also the first eigenvalue of $\tilde{\sig}$. Thus we have 
$\sigma_1(\tilde{\sig})L(\p\tilde{\sig})= (\sigma_1L)^*$. It follows from Theorem \ref{theorem:rotational} that $\tilde{\sig}$ is equivalent to the critical catenoid
in the sense that $T=T(1)$ and $\alpha=1$. This completes the proof. 
\end{proof}

\section{Boundary conformal volume and relative conformal volume} \label{section:cv}

Let $(\sig^k,g)$ be a $k$-dimensional compact Riemaniann manifold with boundary $\p \sig \neq \emptyset$, and let $B^n$ be the unit ball in $\mathbb{R}^n$. 
Assume that $\sig$ admits a conformal map $\vp: \sig \rightarrow B^n$ with $\vp(\p \sig) \subset \p B^n$.
Let $G$ be the group of conformal diffeomorphisms of ${B^n}$. We define the boundary conformal volume to be the Li-Yau \cite{LY} conformal volume of the boundary submanifold $\p \sig$.

\begin{definition}
Given a map $\vp \in C^1(\p \sig, \p B^n)$ that admits a conformal extension $\vp: \sig \rightarrow B^n$, define the {\em boundary $n$-conformal volume of $\vp$} by
\[
      V_{bc}(\sig,n,\vp)=\sup_{f \in G} V ( (f (\vp (\p \sig))).
\]
The {\em boundary $n$-conformal volume of $\sig$} is then defined to be
\[
      V_{bc} (\sig,n)=\inf_{\vp} V_{bc} (\sig,n,\vp)
\]
where the infimum is over all $\vp \in C^1(\p \sig,\p B^n)$ that admit conformal extensions 
$\vp: \sig \rightarrow B^n$.
It can be shown (see Lemma \ref{lemma:dec}) that $V_{bc}(\sig,n) \geq V_{bc} (\sig,n+1)$. The 
{\em boundary conformal volume of $\sig$} is defined to be
\[
      V_{bc} (\sig)=\lim_{n \rightarrow \infty} V_{bc}(\sig, n).
\]
\end{definition}

\begin{remark} \label{rem:bc}
For any $k$-dimensional manifold $\sig$ with boundary, the boundary $n$-conformal volume of $\sig$ is bounded below by the volume of the $(k-1)$-dimensional sphere:
\[
                  V_{bc}(\sig,n) \geq V (S^{k-1}).
\]
The proof is as in \cite{LY}; given a point $\theta$ on $S^{n-1}$, let $f_\theta(t)$ be the one parameter subgroup of the group of conformal diffeomorphisms of the sphere generated by the gradient of the linear functions of $\mathbb{R}^n$ in the direction $\theta$. For all $t$, $f_\theta(t)$ fixes the points $\theta$ and $-\theta$, and $\lim_{t \rightarrow \infty} f_\theta(t)(x)=\theta$ for all $x \in S^{n-1} \setminus \{-\theta\}$. If $\vp: \p \sig \rightarrow S^{n-1}$ is a map whose differential has rank $k-1$ at $x$, then
\[
      \lim_{t \rightarrow \infty} V( f_{-\vp(x)}(t) (\vp(\p \sig))) = m V(S^{k-1} )
\]
for some $m \in \mathbb{Z}^+$ (here the integer $m$ is the multiplicity of the immersed
submanifold $\p\sig$ at the point $-\theta$).
\end{remark}

For $k=2$ and for a minimal surface $\sig$ that is a solution to the free boundary problem in the unit ball $B^n$ in $\mathbb{R}^n$, the boundary $n$-conformal volume of $\sig$ is the length of the boundary of $\sig$; that is, its boundary length is maximal in its conformal orbit.
\begin{theorem} \label{theorem:max}
Let $\sig$ be a minimal surface in $B^n$, with nonempty boundary $\p \sig \subset \p B^n$, and meeting $\p B^n$ orthogonally along $\p \sig$, given by the isometric immersion $\vp: \sig \rightarrow B^n$. Then
\[
       V_{bc}(\sig,n,\vp)= L(\p \sig),
\]
the length of the boundary of $\sig$.
\end{theorem}

\begin{proof}
The trace-free second fundamental form $|| A-\frac{1}{2}(\mbox{Tr}_gA)g ||^2 dV_g$ is conformally invariant for surfaces. Using the Gauss equation we have
$2|| A-\frac{1}{2}(\mbox{Tr}_gA)g ||^2 = H^2 -4K$. Therefore, given any $f \in G$, 
\[    
      \int_{\sig} (H^2-4K) \;da
      =\int_{f(\sig)} (\tilde{H}^2-4\tilde{K}) \;d\tilde{a},
\]
where $d\tilde{a}$ denotes the induced area element on $f(\sig)$, and 
$\tilde{K}$ and $\tilde{H}$ denote the Gauss and mean curvatures of $f(\sig)$ in $\mathbb{R}^n$. Since $\sig$ is minimal, $H=0$, and so we have
\begin{equation} \label{equation:gb}
     -4 \int_{\sig} K \;da
     =\int_{f(\sig)} \tilde{H}^2 \;d\tilde{a} - 4 \int_{f(\sig)} \tilde{K} \;d\tilde{a}.
\end{equation}
By the Gauss-Bonnet Theorem, 
\begin{align*}
     \int_{\sig} K \;da & = 2\pi \chi(\sig) - \int_{\p \sig} \kappa \; ds \\
     \int_{f(\sig)} \tilde{K} \;da & = 2\pi \chi(f(\sig)) - \int_{\p f(\sig)} \tilde{\kappa} \; ds,
\end{align*}
and using this in (\ref{equation:gb}), since $\chi(\sig)=\chi(f(\sig))$, we obtain
\begin{align} \label{eq:kappa}
      4\int_{\p \sig} \kappa \;ds &=\; \int_{f(\sig)} \tilde{H}^2 \; d\tilde{a} 
                                                      + 4 \int_{\p f(\sig)} \tilde{\kappa} \;d\tilde{s} \notag \\
                                                & \geq \; 4\int_{\p f(\sig)} \tilde{\kappa} \; d\tilde{s}.
\end{align}
If $T$ is the oriented unit tangent vector of $\p \sig$, and $\nu$ is the inward unit conormal vector along $\p \sig$, then
\[
   \kappa=\la \frac{dT}{ds},\nu \ra = - \la T, \frac{d\nu}{ds} \ra
   =\la T, \frac{d\vp}{ds} \ra = \la T,T \ra =1,
\]
where in the third to last equality we have used the fact that $\nu=-\vp$ since $\sig$ meets $\p B^n$ orthogonally along $\p \sig$. Since $f$ is conformal, $f(\sig)$ also meets $\p B^n$ orthogonally along $\p f(\sig)$, and so we also have that $\tilde{\kappa}=1$. Using this in (\ref{eq:kappa}) we obtain
\[ 
      L(\p \sig) \geq L(\p f(\sig)).
\] 
This shows that
\[
       L(\p \sig) \geq V_{bc}(\sig,n,\vp)
\]
as claimed.
\end{proof}
     
The proof of Theorem \ref{theorem:max} implies that any minimal surface that is a solution to the free boundary problem in the unit ball in $\mathbb{R}^n$ has area greater than or equal to that of a flat equatorial disk solution.

\begin{theorem} \label{theorem:lowerbound}
Let $\sig$ be a minimal surface in $B^n$, with (nonempty) boundary $\p \sig \subset \p B^n$, and meeting $\p B^n$ orthogonally along $\p \sig$. Then
\[
          2A(\sig)=L(\p\sig) \geq 2\pi.
\]
\end{theorem}

\begin{proof}
Given $f \in G$, as in the proof of Theorem \ref{theorem:max}, we have 
\begin{equation} \label{eq:length}
   L( \p \sig) \geq L(\p f(\sig)).
\end{equation}
Since $\sig$ is minimal, the coordinate functions are harmonic $\Delta_\sig x^i=0$, and 
$\Delta_\sig |x|^2=4$. Therefore,
\[
     4 A(\sig) = \int_\sig \Delta_\sig |x|^2 \;da 
                     = \int_{\p \sig} \fder{|x|^2}{\nu} \;ds
                     =\int_{\p \sig} 2 \;ds
                     =2L(\p \sig).
\]
Using this in (\ref{eq:length}) gives
\[
        2 A(\sig) \geq L(\p f(\sig)).
\]
If $p \in \p \sig$, then as in Remark \ref{rem:bc},
\[
       \lim_{t \rightarrow \infty} L(f_p(t)(\p \sig))=mL(S^1)=2\pi m
\]
for some $m \in \mathbb{Z}^+$, and so, we have the desired conclusion
\[
     2A(\sig)=L(\p\sig) \geq 2\pi.
\]
\end{proof}

\begin{corollary}
The sharp isoperimetric inequality holds for free boundary minimal surfaces in the ball:
\[
         A \leq \frac{L^2}{4 \pi}.
\]
\end{corollary}   

\begin{proof}
For free boundary minimal surfaces in the ball we have $2A(\sig)=L(\p \sig)$, as shown in the proof of Theorem \ref{theorem:lowerbound}. It follows that the inequality $A(\sig) \geq \pi$ is equivalent to the sharp isoperimetric inequality
$A \leq {L^2}/{4 \pi}$.
\end{proof}

\vspace{2mm}
   
\begin{definition}
Let $\sig$ be a $k$-dimensional compact Riemannian manifold with boundary that admits a conformal map $\vp: \sig \rightarrow B^n$ with $\vp(\p \sig) \subset \p B^n$.
Define the {\em relative $n$-conformal volume of $\vp$} by
\[
      V_{rc}(\sig,n,\vp)=\sup_{f \in G} V ( (f (\vp (\sig))).
\]
The {\em relative $n$-conformal volume of $\sig$} is then defined to be
\[
      V_{rc} (\sig,n)=\inf_\vp V_{rc} (\sig,n,\vp)
\]
where the infimum is over all non-degenerate conformal maps $\vp: \sig \rightarrow B^n$ with 
$\vp(\p \sig) \subset \p B^n$.

\begin{lemma} \label{lemma:dec}
If $m \geq n$, then $V_{rc}(\sig,n) \geq V_{rc}(\sig,m)$.
\end{lemma}

\begin{proof}
To see this, suppose $\vp: \sig \rightarrow B^n \subset B^m$ is conformal, with 
$\vp(\p \sig) \subset \p B^n \subset \p B^m$. Let $A=\vp(\sig) \subset B^n$ and suppose that $f$ is a conformal transformation of $B^m$. 
Then $f(A)$ lies in the spherical cap $f(B^n)$ in $B^m$ whose boundary lies in $\p B^m$. 
Let $T \in O(m)$ be an orthogonal transformation that rotates this spherical cap so that its boundary lies in an $n$-plane parallel to the $n$-plane containing the boundary of the original equatorial $B^n$. Let $P$ be the conformal projection of $T(f(B^n))$ onto $B^n$, and let $A'=P(T(f(A)))$. Clearly $P$ is volume increasing, and so
\[
      V(A') \geq V(f(A)).
\]
But $A'$ is the image of $A$ under some conformal transformation of $B^n$, therefore
\[
      \sup_{F \in G} V(F(A)) \geq \sup_{f \in G'} V(f(A)),
\]
where $G$ denotes the group of conformal transformations of $B^n$, and $G'$ denotes the group of conformal transformations of $B^m$.
\end{proof}

The {\em relative conformal volume of $\sig$} is defined to be
\[
      V_{rc} (\sig)=\lim_{n \rightarrow \infty} V_{rc}(\sig, n).
\]
\end{definition}           

\begin{remark}
For any $k$-dimensional manifold $\sig$ with boundary, the relative $n$-conformal volume of $\sig$ is bounded below by the volume of the $k$-dimensional ball:
\[
                  V_{rc}(\sig,n) \geq V (B^k).
\]
To see this, suppose $\vp: \sig \rightarrow B^n$ is a conformal map with 
$\vp(\p \sig) \subset \p B^n$, whose differential has rank $k$ at $x \in \p \sig$. The conformal diffeomorphisms  $f_{-\vp(x)}(t)$ of the sphere (see Remark \ref{rem:bc}), extend to conformal diffeomorphisms of $B^n$, and
\[
      \lim_{t \rightarrow \infty} V( f_{-\vp(x)}(t) (\vp(\sig))) = m V(B^k )
\]
for some $m \in \mathbb{Z}^+$, the multiplicity of $\vp(\p\sig)$ at $\vp(x)$.
\end{remark}

\section{Relationship between the first eigenvalue and conformal volume}

In this section we prove estimates for the first eigenvalue of the Dirichlet-to-Neumann map which are analogs of the estimates of Li and Yau \cite{LY} and El Soufi and Ilias \cite{EI} for the first Neumann eigenvalue of the Laplacian. 

\begin{lemma} \label{lemma:hersch}
Let $(M,g)$ be a compact Riemanian manifold, and let $\vp$ be an immersion of $M$ into $S^{n-1} \subset \mathbb{R}^{n}$. There exists $f \in G$ such that $\psi= f \circ \vp=(\psi^1, \ldots, \psi^{n})$ satisfies
\[
       \int_M \psi^i \; dv_g=0 
\]
for $i=1, \ldots, n$.
\end{lemma}

\begin{proof} See \cite{H}, \cite{LY} page 274. 
\end{proof}

\begin{theorem} \label{thm:rcestimate}
Let $(\sig,g)$ be a compact $k$-dimensional Riemannian manifold with nonempty boundary. Let 
$\sigma_1 >0$ be the first non-zero eigenvalue of the Dirichlet-to-Neumann map on $(\sig,g)$. Then
\[
         \sigma_1 \, V(\p \sig) \, V(\sig)^{\frac{2-k}{k}} \leq k \, V_{rc}(\sig,n)^{\frac{2}{k}}
\]
for all $n$ for which $V_{rc}(\sig,n)$ is defined (i.e. such that there exists a conformal mapping 
$\vp: \sig \rightarrow B^n$ with $\vp(\p \sig) \subset \p B^n$). Equality implies that there exists a conformal harmonic map $\vp: \sig \rightarrow B^n$ which (after rescaling the metric $g$) is an isometry on $\p \sig$, with $\vp(\p \sig) \subset \p B^n$ and such that $\vp(\sig)$ meets $\p B^n$ orthogonally along  $\vp(\p \sig)$. For $k>2$ this map is an isometric minimal immersion of $\Sigma$ to its image.
Moreover, the immersion is given by a subspace of the first eigenspace.
\end{theorem}
The following is an immediate consequence of the theorem.

\begin{corollary} \label{cor:rcestimate-surface}
Let $\sig$ be a compact surface with nonempty boundary and metric $g$. Let $\sigma_1 >0$ be the first non-zero eigenvalue of the Dirichlet-to-Neumann map on $(\sig,g)$. Then
\[
         \sigma_1 \, L(\p \sig) \leq 2 \, V_{rc}(\sig,n)
\]
for all $n$ for which $V_{rc}(\sig,n)$ is defined.
Equality implies that there exists a conformal minimal immersion $\vp: \sig \rightarrow B^n$ by first eigenfunctions which (after rescaling the metric) is an isometry on $\p \sig$, with $\vp(\p \sig) \subset \p B^n$ and such that $\vp(\sig)$ meets $\p B^n$ orthogonally along $\vp(\p \sig)$. 
\end{corollary}

\begin{proof}
Let $\vp : \sig \rightarrow B^n$ be a conformal map with $\vp(\p \sig) \subset \p B^n$.
By Lemma \ref{lemma:hersch} we can assume that $\vp=(\vp^1,\ldots,\vp^n)$ satisfies
\[
       \int_{\p \sig} \vp^i \; ds =0
\]
for $i=1, \ldots, n$.
Let $\hat{\vp}^i$ be a harmonic extension of $\vp^i |_{\p \sig}$. Then,
\begin{equation} \label{eq:var}
     \sigma_1 \leq \frac{\int_\sig |\n \hat{\vp}^i |^2 \; dv_\sig }{\int_{\p \sig} (\vp^i)^2 \; dv_{\p \sig}}
                 \leq \frac{\int_\sig |\n \vp^i|^2 \; dv_\sig }
                                {\int_{\p \sig} (\vp^i)^2 \; dv_{\p \sig}}.
\end{equation}
By H\"{o}lder's inequality, and since $\vp$ is conformal, 
\begin{align*}
      \int_\sig \sum_{i=1}^n| \n \vp^i |^2 \; dv_\sig 
      &\leq V(\sig)^{\frac{k-2}{k}} 
         \left[ \int_\sig 
         \Big( \sum_{i=1}^n| \n \vp^i |^2 \Big)^{\frac{k}{2}} \; dv_\sig \right]^{\frac{2}{k}} \\
      &= V(\sig)^{\frac{k-2}{k}} \left[ k^{\frac{k}{2}} V( \vp (\sig)) \right]^{\frac{2}{k}} \\
      &\leq kV(\sig)^{\frac{k-2}{k}}  V_{rc}(\sig,n,\vp)^{\frac{2}{k}}. \\
\end{align*}
On the other hand, since $\vp(\p \sig) \subset \p B^n$, 
\[
     \sum_{i=1}^n \int_{\p \sig} (\vp^i)^2 \; dv_{\p \sig} = \int_{\p \sig} dv_{\p \sig} = V(\p \sig).
\]
Then by (\ref{eq:var}) we have
\[
     \sigma_1 \, V(\p \sig) \, V(\sig)^{\frac{2-k}{k}} \leq k V_{rc} (\sig,n,\vp)^{\frac{2}{k}}.
\]
Since $V_{rc}(\sig,n)=\inf_\vp V_{rc} (\sig,n,\vp)$ we get
\[
     \sigma_1 \, V(\p \sig) \, V(\sig)^{\frac{2-k}{k}} \leq k V_{rc} (\sig,n)^{\frac{2}{k}}.
\]

Now assume that we have equality, 
$\sigma_1 \, V(\p \sig) = k V_{rc} (\sig,n)^{2/k} \, V(\sig)^{(k-2)/k}$.
Choose a sequence of conformal maps $\vp_j: \sig \rightarrow B^n$ with $\vp_j(\p \sig) \subset \p B^n$, such that 
\[
      \lim_{j \rightarrow \infty} V_{rc} (\sig, n, \vp_j) = V_{rc}(\sig,n)
\]
and by composing with a conformal transformation of the ball we may assume
\[
     \int_{\p \sig} \vp_j^i \; ds =0
\]
for all $i,\,j$. By changing the order of coordinates, we may assume that
\[
   \lim_{j \rightarrow \infty} \int_\sig (\vp_j^i)^2 \; da \quad
   \begin{cases}
        > 0 & i=1, \ldots, N \\
        =0  & i= N+1, \ldots, n.
   \end{cases}
\]
We have
\begin{align*}
        \sigma_1V(\p \sig) 
        & = \sigma_1 \sum_{i=1}^n \int_{\p \sig} (\vp_j^i)^2 \; dv_{\p \sig}
            \leq \sum_{i=1}^n \int_\sig |\n \vp_j^i |^2 \; dv_\sig \\ 
        & \leq V(\sig)^{\frac{k-2}{k}} \left[ \int_\sig 
                  \Big( \sum_{i=1}^n| \n \vp_j^i |^2 \Big)^{\frac{k}{2}} \; dv_\sig \right]^{\frac{2}{k}} 
         \leq k V_{rc} (\sig,n,\vp_j)^{\frac{2}{k}} V(\sig)^{\frac{k-2}{k}}.
\end{align*}
Letting $j \rightarrow \infty$ and using 
$\sigma_1 \, V(\p \sig) = k V_{rc} (\sig,n)^{2/k} \, V(\sig)^{(k-2)/k}$ 
we get
\begin{align}\label{eq:limit}
       \sigma_1V(\p \sig) 
         & =\sigma_1 \lim_{j \rightarrow \infty} \sum_{i=1}^n \int_{\p \sig} (\vp_j^i)^2 \; dv_{\p \sig}
             =\lim_{j \rightarrow \infty} \sum_{i=1}^n \int_\sig |\n \vp_j^i |^2 \; dv_\sig \notag \\
         & =V(\sig)^{\frac{k-2}{k}} \lim_{j \rightarrow \infty}\left[ \int_\sig 
         \Big( \sum_{i=1}^n| \n \vp_j^i |^2 \Big)^{\frac{k}{2}} \; dv_\sig \right]^{\frac{2}{k}}
         =\sigma_1 V(\p \sig).
  \end{align}                              
Therefore, for any fixed $i$, $\{\vp_j^i\}$ is a bounded sequence in $W^{1,k}(\sig,\mathbb{R})$, and
since the inclusion $W^{1,k}(\sig,\mathbb{R}) \subset L^2(\sig,\mathbb{R})$ is compact, by passing to a subsequence we can assume that $\{\vp_j^i\}$ converges weakly in $W^{1,k}(\sig,\mathbb{R})$, strongly in $L^2(\sig,\mathbb{R})$, and pointwise a.e., to a map $\psi^i: \sig \rightarrow \mathbb{R}$. Clearly
$\sum_{i=1}^n (\psi^i)^2 \leq 1$ a.e. on $\sig$, $\sum_{i=1}^n (\psi^i)^2=1$ a.e. on $\p \sig$, and 
$\psi^i=0$ for $i=N+1, \ldots, n$. Since for all $i$
\[
     \sigma_1 \int_{\p \sig} (\vp_j^i)^2 \; dv_{\p \sig} \leq \int_\sig |\n \vp_j^i |^2 \; dv_\sig
\]
and
\[
     \sigma_1 \lim_{j \rightarrow \infty} \sum_{i=1}^n \int_{\p \sig} (\vp_j^i)^2 \; dv_{\p \sig}
     = \lim_{j \rightarrow \infty} \sum_{i=1}^n \int_\sig |\n \vp_j^i|^2 \; dv_\sig,
\]
we have
\begin{equation} \label{eq:psi}
   \lim_{j \rightarrow \infty} \int_\sig |\n \vp_j^i |^2 \; dv_\sig
     = \sigma_1 \lim_{j \rightarrow \infty} \int_{\p \sig} (\vp_j^i)^2 \; dv_{\p \sig}       
     = \sigma_1 \int_{\p \sig} (\psi^i)^2 \; dv_{\p \sig}
     \leq  \int_\sig |\n \psi^i|^2 \; dv_\sig.
\end{equation}
On the other hand, $\vp_j^i \rightarrow \psi^i$ weakly in $W^{1,k}(\sig, \mathbb{R})$, and so
\[
     \int_\sig |\n \psi^i|^2 \; dv_\sig \leq \lim_{j \rightarrow \infty} \int_\sig |\n \vp_j^i |^2 \; dv_\sig.
\]
Therefore, we must have equality in (\ref{eq:psi}), and so 
\[
      \lim_{j \rightarrow \infty} \int_\sig |\n \vp_j^i |^2 \; dv_\sig = \int_\sig |\n \psi^i|^2 \; dv_\sig
\]
which means $\{\vp_j^i\}$ converges to $\psi$ strongly in $W^{1,2}(\sig,\mathbb{R})$.
Moreover, 
\[
      \sigma_1 \int_{\p \sig} (\psi^i)^2 \; dv_{\p \sig} = \int_\sig |\n \psi^i|^2 \; dv_\sig
\]
and it follows that $\{ \psi^i \}_{i=1}^N$ are first eigenfunctions. In particular, $\psi^i$ is harmonic for $i=1, \ldots, N$. Also, since $\vp_j$ is conformal and converges strongly in $W^{1,2}$ to $\psi$, the map
\begin{align*}
       \psi: \sig & \rightarrow B^N \\
       x &\mapsto (\psi^1(x), \ldots, \psi^N(x))
\end{align*}
defines a conformal map.  Therefore, $\psi: \sig \rightarrow B^N$ is conformal and harmonic, with 
$\psi(\p \sig) \subset \p B^N$. 
Since $\psi(\p \sig) \subset \p B^N$ and
\begin{equation} \label{eq:ev}
      \fder{\psi}{\nu}=\sigma_1 \psi
\end{equation}
on $\p \sig$ since $\psi^i$ are eigenfunctions, it follows that $\psi(\sig)$ meets $\p B^N$ orthogonally along $\psi(\p \sig)$. 

By scaling the metric we can assume that $\sigma_1=1$. Then by (\ref{eq:ev}), on $\p \sig$ we have
\[
      \left|\fder{\psi}{\nu}\right|=|\psi|=1,
\]
and hence $\psi$ is an isometry on $\p \sig$. Finally, for $k>2$ we have from (\ref{eq:limit})
\[ 
        \lim_{j \rightarrow \infty} \sum_{i=1}^n \int_\sig |\n \vp_j^i |^2 \; dv_\sig
      =\sum_{i=1}^n \int_\sig |\n \psi^i |^2 \; dv_\sig
      =V(\sig)^{\frac{k-2}{k}} \lim_{j \rightarrow \infty}\left[ \int_\sig 
         \Big( \sum_{i=1}^n| \n \vp_j^i |^2 \Big)^{\frac{k}{2}} \; dv_\sig \right]^{\frac{2}{k}}.
\]
By lower semicontinuity of the norm under weak convergence this implies
\[
        \int_\sig |\n \psi|^2 \; dv_\sig
      \geq V(\sig)^{\frac{k-2}{k}} \left[ \int_\sig 
         \Big( \sum_{i=1}^n| \n \psi^i |^2 \Big)^{\frac{k}{2}} \; dv_\sig \right]^{\frac{2}{k}}.
\]
Now the H\"older inequality implies the opposite inequality and thus we have equality
in the H\"older inequality, which implies $|\n\psi|^2$ is constant on $\Sigma$, and this constant must
be $k$ by the boundary normalization. Since $\psi$ is conformal this implies that $\psi$ is
an isometry as claimed. 

\end{proof}

\bibliographystyle{plain}

\end{document}